\documentclass[a4paper,10pt]{amsart}
\usepackage[utf8]{inputenc}
\usepackage{amssymb}
\usepackage{amsmath}
\usepackage{amsfonts}

\newtheorem{theorem}{Theorem}[section]
\newtheorem{lemma}[theorem]{Lemma}
\newtheorem{corollary}[theorem]{Corollary}
\newtheorem{proposition}[theorem]{Proposition}
\newtheorem{definition}[theorem]{Definition}

\def\R{\mathbb{R}}

\def\set#1{\left\{\, #1 \,\right\}}
\providecommand{\abs}[1]{\left\vert \,#1\, \right\vert}
\providecommand{\norm}[1]{\left\| \,#1\, \right\|}
\providecommand{\inner}[2]{\left\langle\,#1,#2\,\right\rangle}

\def\H{\mathbb{H}}

\def\calH{\mathcal{H}}
\def\calM{\mathcal{M}}

\title[Minimizing configurations and Hamilton-Jacobi
equations]{Minimizing configurations and Hamilton-Jacobi
equations of homogeneous N-body problems}
\author[E. Maderna]{Ezequiel Maderna}
 \address{Centro de Matemática\\
    Universidad de la República, Montevideo, Uruguay.}
 \email{emaderna@cmat.edu.uy}

\begin{document}

\begin{abstract}
For $N$-body problems with homogeneous potentials
we define a special class of central configurations
related with the reduction of homotheties in the study
of homogeneous weak KAM solutions. For potentials in
$1/r^\alpha$ with $\alpha\in (0,2)$ we prove the existence
of homogeneous weak KAM solutions. We show that such
solutions are related to viscosity solutions of another
Hamilton-Jacobi equation in the sphere of normal configurations.
As an application we prove for the Newtonian
three body problem that there are no
smooth homogeneous solutions to the critical
Hamilton-Jacobi equation.
\end{abstract}

\maketitle

\section{Introduction.}

We consider N-body problems with homogeneous potentials
\[U(x)=\sum_{i<j}\frac{m_i\,m_j}{r_{ij}\,^{2\kappa}}\]
where $x=(r_1,\dots,r_N)\in E^N$ is a configuration
of the $N$ massive punctual bodies in some Euclidean space
$E$, the positive constants $m_i$ are their respective
masses, and $r_{ij}=\norm{r_i-r_j}$. The case in which
$E=\R^3$ and $\kappa=1/2$
is the classical Newtonian N-body problem. An old and
natural question related to the study of the dynamics
of such systems is the complete integrability. When the
problem is completely integrable, the phase space can be
foliated by invariant Lagrangian manifolds, and each
leaf must then be contained in a level set of the energy function.
In particular, if one of these Lagrangian manifolds
is a graph over the space of configurations, it must
correspond via the Legendre transformation with the
graph of a closed 1-form $\omega$ which satisfies
$H(x,\omega_x)=c$ for some constant $c\in\R$.
Since the Hamiltonian of the system is the function
on the cotangent bundle of the configuration space
given by
\[H(x,p)=\frac{1}{2}\norm{p}^2-U(x)\]
and $\inf U(x)=0$,
we have that the last equation can not be solved
if $c<0$. Of course, the Hamiltonian is finite only
in the open and dense set of configurations without
collisions
\[\Omega=\set{x=(r_1,\dots,r_N)\in E^N\mid
r_i=r_j \iff i=j}\]
which is simply connected when $\dim E\geq 3$. Therefore,
these considerations lead us to investigate the existence of
global solutions of the Hamilton-Jacobi equation
$H(x,d_xu)=c$ for $c\geq 0$. In this paper we study
the critical case,
\begin{equation}\label{HJ}
\norm{d_xu}^2=2\,U(x)
\end{equation}
and especially, the existence of global
homogeneous solutions for this equation.
In all what follows, the norm of a configuration
in $E^N$ will be the norm associated to the mass
inner product, and $(E^N)^*$ will be endowed with the
corresponding dual norm.

The author has showed in \cite{Mad1},
for values of $\kappa\in (0,1)$, the existence of
global viscosity solutions to the Hamilton-Jacobi equation
\eqref{HJ} using weak KAM theory. Moreover, there are
invariant solutions with respect to the obvious action of the
compact Lie group $O(E)$ in $E^N$. In \cite{Mad2} it was
proved that, in the Newtonian case, every weak solution of
\eqref{HJ} is invariant with respect to the action
by translations of $E$ in $E^N$.
That is to say, each solution of \eqref{HJ} is uniquely determined
by his restriction to the subspace $V$ of $E^N$ of configurations with center
of mass at $0\in E$. Thus, a configuration $x=(r_1,\dots,r_N)\in E^N$ is
in $V$ if and only if $\sum_{i=1}^N m_i\,r_i=0$.
We also have that
\[V=\Delta^\perp=\set{(r,\dots,r)\mid r\in E}^\perp\]
where the orthogonal complement is taken with respect to the
mass inner product in $E^N$.
Moreover, the translation invariance of a given function
$f:E^N\to\R$ implies that at each point of differentiability
$x\in E^N$ we have
\[d_{p(x)}(f\mid_V)=(d_xf)\mid_V\]
where $p:E^N\to V$ is the orthogonal projection on $V$ (in other words,
$p(x)$ is the unique translation of $x$ with center of mass at $0\in E$).
Therefore we have:

\vspace{1mm}
\emph{A function $u:E^N\to\R$ is a solution (in any possible way)
of Hamilton-Jacobi equation \eqref{HJ} if and only if his restriction
$u\mid_V:V\to R$ is a solution of the same Hamilton-Jacobi equation in $V$.}
\vspace{1mm}

For this reason, in all what follows we will only consider configurations
in $V$, and functions $u:V\to\R$. As a subspace of $E^N$, $V$ is also
an Euclidean space with the mass inner product, and his dual space
will be considered with the corresponding dual norm.

On the other hand,
non rotation invariant weak solutions can exist, and
a simple example for the planar Kepler problem was suggested
by Alain Chenciner and the author, and later it was found
explicitly by Andrea Venturelli
(We will explain more about these examples in \cite{VentuMad2}).
More or less at the same time, Alain Chenciner asked
if all these weak solutions are necessarily homogeneous functions
modulo a constant.
Again using the non rotation invariant examples, as well as
some characteristic properties of weak KAM solutions,
non homogeneous weak solutions can be constructed.
Here we will prove the existence of homogeneous weak
KAM solutions (theorem \ref{existHweakKAM} bellow).

Until now the most fruitful application of the existence
results of weak KAM solutions is that each one of them
gives rise to a lamination of the configuration space by
completely parabolic motions, showing the abundance of such
motions (see \cite{daLuzMad,VentuMad}).
Moreover, the associated lamination
defines the solution up to a constant, thus it is natural
to expect that invariance properties of the solutions can
be expressed in terms of properties of the lamination.

Let $I(x)$ be the moment of inertia of a configuration
$x=(r_1,\dots,r_N)$ with respect to the origin of $E$,
that is to say,
\[I(x)=\sum_{i=1}^Nm_i\,r_i^2\]
and let $S=\set{x\in V\mid I(x)=1}$ be the sphere of
normal configurations. In fact, $I$ is the quadratic form
associated to the mass inner product in $V$ and $S$ is
the unit sphere in $V$ of the corresponding norm.
Every configuration $x\neq 0$ has a unique polar decomposition,
namely $x=\lambda s$, with $\lambda>0$ and $s\in S$.
Therefore an homogeneous function $u:V\to\R$
of degree $\alpha$ is uniquely determined by his restriction
$v$ to the unit sphere $S$, since we must have
$u(\lambda\,s)=\lambda^\alpha\,v(s)$.
We will show that, the equation must satisfy
the function $v$ in order to be
an homogeneous solution of \eqref{HJ}
the function $u$, is the
Hamilton-Jacobi equation
\begin{equation}\label{HJH}
(1-\kappa)^2\,v(s)^2+\norm{d_sv}^2=2\,U(s)\,.
\end{equation}
Note that this equation is not the Hamilton-Jacobi
equation arising from a Tonelli Hamiltonian.
However we will deduce for $\kappa\in (0,1)$
the existence of global viscosity solutions of
Hamilton-Jacobi equation \eqref{HJH}, see theorem
\ref{existsViscSol2} bellow.

Following the analogy with the Aubry-Mather theory
we define a special type of central configurations,
and we will prove that they are intimately related
with the solutions of \eqref{HJH}. Recall that
a free time minimizer is a curve whose
restriction to any compact interval minimizes the
Lagrangian action in the set of all curves with the same
extremities (see \cite{daLuzMad} for a detailed description
of this concept in the Newtonian case).
They correspond to the semistatic curves in the Aubry-Mather
theory and must have critical energy (zero energy in our case).
We also recall that a central
configuration is a configuration $x\in V$ for which
there are homothetic motions passing through it. If that
is the case, then only two (modulo translation of time)
of these homothetic motions
have zero energy, namely the parabolic ejection and the parabolic
collision by $x$.

\begin{definition}
A \emph{minimizing configuration} is a central configuration
such that the corresponding parabolic ejection is a free time minimizer.
We will denote $\calM$ the set of normal minimizing configurations
(i.e. in the sphere $S$).
\end{definition}

The set of minimizing configurations is not empty. Note that the
potential $U$ has a minimum on the sphere $S$. If we denote
\[U_0=\min\set{U(x)\mid x\in S}\]
and
\[\calM_0=\set{x\in S\mid U(x)=U_0}\]
then it is clear that $\calM_0$ is not empty. We will call
$\calM_0$ the set of normal \emph{minimal configurations}.
It is not difficult to prove that we have
\[\calM_0 \subset \calM\]
or in other words, that every minimal configuration is minimizing.
An easy proof of this fact in the Newtonian case can be found in
\cite{daLuzMad} (proposition 3.4), and the same proof works for
any homogeneous potential with minor changes of the constants.

In the context of the
Aubry-Mather theory, there is a well known conjecture due
to Ricardo Mañé which says that, for a generic Tonelli Lagrangian
on a given closed manifold $M$, the Mather set is reduced to
an hyperbolic periodic orbit of the Lagrangian flow on the
tangent bundle $TM$.
When this happens all the theory becomes simple, for instance
the semistatic curves are exactly the projection on $M$ of orbits
in the stable manifold of the Mather set, and there is only one
weak KAM solution modulo a constant.
Several years ago Renato Iturriaga asked to the author if
there is an analogous conjecture in our context. He proposed
that the correct conjecture must be: \emph{For generic values
of the masses, there is only one normal minimal configuration
(modulo isometries) and it is non degenerate}. Of course,
here the non degeneracy refers to the transversal directions
to the action of the orthogonal group $O(E)$ on the sphere $S$.
Now we can see that the interesting question is to determine
if the same happens with the minimizing configurations instead
of the minimal ones. Note that the uniqueness of
minimizing configuration implies $\calM=\calM_0$, but
the only known results which proves this equality was obtained
for some homogeneous $N$-body problems by Barutello and Secchi
in \cite{Vivina}. Using a variational Morse-like index they prove
that several colliding trajectories are not minimizing.
In particular they prove for the three body problem with
equal masses that the collinear central configurations are not minimizing
whenever $\kappa>3-2\sqrt{2}$ (note that the Newtonian case is included).
Thus in these case the only minimizing configurations are the Lagrange
equilateral configurations. In order to show the interest of
this analysis, we will prove in section \ref{sec:smoothHsol}
the following theorem.

\begin{theorem}\label{teo:3BP no smooth homogeneous}
The critical Hamilton-Jacobi equation of the Newtonian three body problem
has no smooth homogeneous solutions.
\end{theorem}

\section{The Lax-Oleinik semigroup and weak KAM solutions.}

We need to recall briefly some facts about the Lagrangian
action, and variational properties related to
the Hamilton-Jacobi equation. Also we recall
the definition of the Lax-Oleinik semigroup whose
fixed points are precisely the weak KAM solutions.
The proofs of the statements below can be found
in \cite{Mad1}.

If $\gamma:[a,b]\to V$ is an absolutely continuous curve,
then the Lagrangian action of $\gamma$ is the value in
$(0,+\infty]$ defined by
\[A(\gamma)=\int_a^b \frac{1}{2}\,\dot\gamma(t)^2+
U(\gamma(t))\,dt\]
where the square of the vector $\dot\gamma(t)$, which is
defined for almost every $t\in [a,b]$, is taken with
respect to the mass inner product in $V$. It is well
known that if such a curve has a finite action, then it
must be in the Sobolev space $H^1([a,b],V)$.
Of course, since our system is autonomous, each curve
can be parameterized in an interval of the form $[0,t]$
by translation in time, and preserving his action.
We will denote $\phi(x,y,t)$ the infimum of the Lagrangian
action in the set of all curves going from $x$ to $y$
in time $t>0$. The infimum without restriction of time
will be denoted $\phi(x,y)$.
We know that for each $\kappa\in (0,1)$, there is a
positive constant $\eta>0$ such that the inequality
\begin{equation}\label{holder}
\phi(x,y)=\inf_{t>0}\phi(x,y,t)
\leq \eta\, \norm{x-y}^{1-\kappa}
\end{equation}
holds for any pair of configurations $x,y\in E^N$.

The set of weak subsolutions of Hamilton-Jacobi
equation \eqref{HJ} is
\[\calH=\set{u:V\to\R\mid u(x)-u(y)\leq \phi(x,y)
\textrm{ for all }x,y\in V}\]
and will be endowed with the topology
of uniform convergence on compact subsets.
Since there is a trivial action of $\R$ in $\calH$
given by addition of constants, we can deduce that
$\calH$ is homeomorphic to $\R\times \calH_0$, where
\[\calH_0=\set{u\in\calH\mid u(0)=0}\]
is a compact set of functions because of the
Hölder estimate \eqref{holder}.
The set of weak subsolutions $\calH$ is clearly convex.
Another interesting property of $\calH$ is that it
contains the infimum of any family of his elements
whenever the infimum is finite.

\begin{lemma}\label{infimoSUBsol}
If $F\subset\calH$ is such that
$u_F(x_0)=\inf\set{u(x_0)\mid u\in F}>-\infty$
for some $x_0\in E^N$, then
$u_F(x)=\inf\set{u(x)\mid u\in F}$ is finite
at every configuration $x\in E^N$ and defines a
weak subsolution $u_F\in\calH$.
\end{lemma}

\begin{proof}
Let $x\in V$ be any configuration. Since
$F\subset\calH$, for each $u\in F$
we have
$u(x)\geq u(x_0)-\phi(x_0,x)$, thus
\[u(x)\geq u_F(x_0)-\phi(x_0,x)\]
which implies that $u_F(x)\geq -\infty$ and
that $u_F(x_0)-u_F(x)\leq \phi(x_0,x)$.
Replacing now $x$ and $x_0$ by any two
configurations $x$ and $y$ in the previous
argument we conclude that $u_F\in\calH$.
\end{proof}

The action in $\calH$ of the Lax-Oleinik semigroup
$(T_t)_{t\geq 0}$ is given by
\[T_tu(x)=\inf\set{u(y)+\phi(y,x,t)\mid y\in V}\]
for $t>0$, and $T_0u=u$ for all $u\in\calH$.
Note that we have
\[\calH=\set{u:V\to\R\mid
u\leq T_tu\textrm{ for all }t\geq 0}\,.\]
Note also that if $u_1,u_2\in\calH$, and
$u_1\leq u_2$, then $T_tu_1\leq T_tu_2$.
The action of the Lax-Oleinik semigroup is continuous,
and the weak KAM theorem says that the set of fixed
points is not empty.

\begin{definition}
A function $u:V\to\R$ is called a \emph{weak KAM solution} if
it is a fixed point of the Lax-Oleinik semigroup
($u=T_tu$ for all $t\geq 0$).
\end{definition}

Weak KAM solutions are viscosity solutions of
Hamilton-Jacobi equation \eqref{HJ}, a notion of
weak solution that we will recall in section
\ref{visco}. They can be characterized
between weak subsolutions as follows:

\begin{proposition}\label{calibrating}
A function $u$ is a weak KAM solution if and only if
\begin{enumerate}
  \item[1.] $u\in\calH$
  \item[2.] Given $x\in V$ there is a curve $\gamma$
            defined for $t\leq 0$ such that
            \begin{enumerate}
             \item $\gamma(0)=x$
             \item $u(x)-u(\gamma(t))=A(\gamma\mid_{[t,0]})$
                for all $t\leq 0$.
            \end{enumerate}
\end{enumerate}
\end{proposition}
Note that the curves $\gamma$ in the above proposition
are free time minimizers of the Lagrangian action because
we have
\[\phi(x,\gamma(t))\geq u(x)-u(\gamma(t))=
A(\gamma\mid_{[t,o]})\geq \phi(x,\gamma(t))\]
for all $t\leq 0$. In the Newtonian case it was proved in
\cite{daLuzMad} that they are motions of zero energy
and completely parabolic (for $t\to-\infty$).

\begin{definition}\label{def:smoothsol}
We will say that a function $u:V\to\R$ is a \emph{smooth solution} of
Hamilton-Jacobi equation \eqref{HJ} if it is differentiable
and satisfy the equation at every configuration $x$ such that
$U(x)<+\infty$ (at configurations $x$ without collisions).
\end{definition}

In the collinear case, that is when $\dim E=1$, we can have
discontinuous smooth solutions. The reason for this is that the
set of configuration without collisions has $n!$ connected components
and we can add to a given solution a different constant on each
component, which results in a new solution. When minimizing curves
avoid collisions (as happens in the Newtonian case, see \cite{Marchal})
we can deduce that a smooth solution $u$ must be a weak subsolution,
and must satisfy the Hölder estimate \eqref{holder}. Therefore,
in this case smooth solutions must be Hölder continuous at collision
configurations.
On the other hand, the differentiability of a given weak KAM solution
at some configuration without collisions $x\in V$ is equivalent
to the uniqueness of the calibrating curve given by proposition
\ref{calibrating}. This fact is of local nature and the proof
can be found in \cite{Fathibook}. The notion of calibrating
curve appears several times in what follows, and for this reason
we will now give a more general definition.

\begin{definition}
Given a weak subsolution $u\in\calH$ and a curve $\gamma:I\to V$,
we say that $\gamma$ calibrates $u$ if we have
$u(\gamma(b))-u(\gamma(a))=A(\gamma\mid_{[a,b]})$
whenever $[a,b]\subset I$.
\end{definition}

\section{Homogeneous solutions}

\subsection{Preliminaries and existence of weak homogeneous solutions}

Suppose that $u\in\calH$ is an homogeneous function
of degree $\alpha$ (for example every constant function
is in $\calH$ and homogeneous of degree $0$).
It is clear that if $u$ is differentiable at
some configuration $x$, and $\lambda>0$, then
$u$ is also differentiable at the configuration
 $\lambda\,x$ and $d_{\lambda\,x}u=\lambda^{\alpha-1}d_xu$.
Therefore if $u$ is a solution of Hamilton-Jacobi
equation \eqref{HJ} we must have
$\norm{d_xu}^2=2\,U(x)$,
and also
\[\lambda^{2(\alpha-1)}\,\norm{d_xu}^2=
\norm{d_{\lambda\,x}u}^2=2\,U(\lambda\,x)=
2\,\lambda^{-2\kappa}\,U(x)\]
from which we get that the degree of
homogeneity must be $\alpha=1-\kappa$.

Homogeneous functions can also be viewed as fixed
points of an action of the multiplicative group $\R^+$.
More precisely, for $\lambda>0$ and $u:V\to\R$
is a given function, we
can define the function $S_\lambda u$ by
\[S_\lambda u(x)=\lambda^{\kappa-1}\,u(\lambda\,x)\]
which defines the group action. Therefore,
a function $u$ is homogeneous
of degree $1-\kappa$ if and only if
$S_\lambda u=u$ for every $\lambda>0$.

When we reduce the rotational symmetries, one of
the main tools involved is the commutation of the
$O(E)$ action with the Lax-Oleinik semigroup.
For every pair of configurations
$x,y\in V$, for every $g\in O(E)$, and
for every $t>0$, we have that $\phi(gx,gy,t)=\phi(x,y,t)$.
Therefore, using the notation $gu$ for $u\circ g$,
we can write $T_t(gu)=g\,T_tu$ for every function
$u:V\to\R$ and every $t\geq 0$.
In particular, if $u\in\calH$, we have $u\leq T_tu$ for
all $t\geq 0$, hence
$gu\leq g\,T_tu=T_t(gu)$
which says that $gu\in\calH$. This also implies
that the set of invariant functions is preserved
by the Lax-Oleinik semigroup: if $u=gu$ then
$T_tu=T_t(gu)=g\,T_tu$. By this way we get that
the set of invariant weak KAM solutions is not empty.

We return now our attention to the reduction
of homotheties.
It is not difficult to see that the
group $(S_\lambda)_{\lambda>0}$ preserves
the set of weak subsolutions $\calH$
as well as the set of weak KAM solutions.
We will need the following lemma.

\begin{lemma}\label{homogeneityPHIxyt}
Given $x,y\in V$, $t>0$, and $\lambda>0$, we have
\[\phi(\lambda x,\lambda y, \lambda^{1+\kappa}t)
=\lambda^{1-\kappa}\,\phi(x,y,t)\,.\]
\end{lemma}

\begin{proof}
Let $\gamma:[0,t]\to V$ be an absolutely continuous curve
such that $\gamma(0)=x$ and $\gamma(t)=y$. Define the
curve $\gamma_\lambda$ on the interval $[0,\lambda^{1+\kappa}t]$
by
\[\gamma_\lambda(s)=\lambda\,\gamma(\lambda^{-(1+\kappa)}s)\,.\]
A simple computation shows that
$A(\gamma_\lambda)=\lambda^{1-\kappa}A(\gamma)$. Note that
the curve $\gamma_\lambda$ goes from $\lambda x$ to $\lambda y$.
Taking a minimizing sequence for the Lagrangian action
in the set of curves going from $x$ to $y$ in time $t$ we
deduce that the inequality
\[\phi(\lambda x,\lambda y, \lambda^{1+\kappa}t)
\leq\lambda^{1-\kappa}\,\phi(x,y,t)\]
is always verified. Therefore the reverse inequality
is also verified, since we have
\[\phi(x,y,t)=
\phi(\lambda^{-1}\lambda x,\lambda^{-1}\lambda y,
\lambda^{-(1+\kappa)}\lambda^{(1+\kappa)}t)\leq
\lambda^{-(1-\kappa)}
\phi(\lambda x,\lambda y, \lambda^{1+\kappa}t)\,.\]
\end{proof}

We will see now that, although the two actions do not
commute, there is a natural relation between them.

\begin{proposition}\label{commute}
For any $u:V\to\R$, $\lambda>0$ and $t\geq 0$ we have
\[T_t\,S_\lambda u = S_\lambda\,T_{\lambda^{(1+\kappa)}t}u\,.\]
\end{proposition}

\begin{proof}
For each $x\in V$ we have
\begin{eqnarray*}
T_t\,S_\lambda u (x)&=&
\inf\set{\lambda^{\kappa-1} u(\lambda y)+\phi(y,x,t)\mid y\in V}\\
&=&\lambda^{\kappa-1}\inf\set{u(\lambda y)+
\lambda^{1-\kappa}\phi(y,x,t)\mid y\in V}\\
&=&\lambda^{\kappa-1}\inf\set{u(\lambda y)+
\phi(\lambda y,\lambda x,\lambda^{1+\kappa}t)\mid y\in V}\\
&=&\lambda^{\kappa-1}\,T_{\lambda^{(1+\kappa)}t}u(\lambda x)=
S_\lambda\,T_{\lambda^{(1+\kappa)}t}u (x)
\end{eqnarray*}
\end{proof}

From this relation we can deduce that the
group $(S_\lambda)_{\lambda>0}$ preserves both
the set of weak subsolutions
and the set of weak KAM solutions.

\begin{corollary}
If $u\in\calH$ then $S_\lambda u\in\calH$ for all $\lambda>0$.
\end{corollary}

\begin{proof}
Fix $\lambda>0$ and suppose that $u\leq T_tu$ for any $t\geq 0$.
Therefore we have
$S_\lambda u \leq S_\lambda\,T_tu$. Using proposition \ref{commute}
we get that
\[S_\lambda u \leq T_{\lambda^{-(1+\kappa)}t}\,S_\lambda u\]
which is equivalent to say that $S_\lambda u\leq T_t\,S_\lambda u$
for any $t\geq 0$.
\end{proof}

\begin{corollary}\label{SpreserveWKAM}
If $u\in\calH$ is a weak KAM solution, then $S_\lambda u$ is also
a weak KAM solution for any $\lambda>0$.
\end{corollary}

\begin{proof}
Fix $\lambda>0$ and $u\in\calH$ such that $u=T_tu$ for any $t\geq 0$.
Therefore we have
\[T_t\,S_\lambda u=S_\lambda\,T_{\lambda^{(1+\kappa)}t}u=
S_\lambda u\]
for any $t\geq 0$, which says that $S_\lambda u$ is a weak
KAM solution.
\end{proof}

Now we are able to prove our first existence result.

\begin{theorem}\label{existHweakKAM}
If $\kappa\in (0,1)$ the set of homogeneous weak KAM
solutions of the $N$-body problem with homogeneous potential
of degree $-2\kappa$ is not empty.
\end{theorem}

\begin{proof}
From the weak KAM theorem proved in \cite{Mad1} we know
that for $\kappa\in(0,1)$ there exists a weak KAM solution
$u\in\calH$. Moreover, adding a constant to $u$ we can assume
that $u\in\calH_0$, that is to say, that $u(0)=0$.
Since $S_\lambda u(0)=0$ for every $\lambda>0$, we can
apply lemma \ref{infimoSUBsol} to define $u_0\in\calH$ as
\[u_0=\inf_{\lambda>0}S_\lambda u\,.\]
Thus we have $u_0\leq T_tu_0$ for all $t\geq0$. On the other
hand, since for each $\lambda>0$ we have $u_0\leq S_\lambda u$,
we also have $T_tu_0\leq T_t\,S_\lambda u$, and we deduce
that
\[T_tu_0\leq \inf_{\lambda>0}T_t\,S_\lambda u\,.\]
Therefore lemma \ref{SpreserveWKAM} implies that
$T_tu_0\leq u_0$ for all $t\geq 0$. We have proved that $u_0$ is a
weak KAM solution. It remains to prove that $u_0$
is homogeneous. For each $\eta>0$ we have
\begin{eqnarray*}
S_\eta u_0&=&\eta^{\kappa-1}\inf_{\lambda>0}S_\lambda u(\eta x)\\
&=&\eta^{\kappa-1}\inf_{\lambda>0}\lambda^{\kappa-1} u(\lambda\,\eta x)\\
&=&\inf_{\lambda>0}(\lambda\,\eta)^{\kappa-1} u(\lambda\,\eta x)\\
&=&\inf_{\lambda>0}S_{\lambda\,\eta} u(x)=u_0(x)
\end{eqnarray*}
which prove that $u_0$ is homogeneous.
\end{proof}

\subsection{Hamilton-Jacobi equation on the sphere.}

Let $u:V\to\R$ be a smooth solution of Hamilton-Jacobi
equation \eqref{HJ}. Let $v:S\to\R$ be the restriction
of $u$ to the unit sphere $S$. If $u$ is homogeneous we have
\[u(\lambda\,s)=\lambda^{1-\kappa}v(s)\]
for all $s\in S$ and all $\lambda>0$. Note that the Riemannian metric
given by the mass inner product in $V$ splits in polar
coordinates $(\lambda,s)$ as
\[dx^2=d\lambda^2+\lambda^2\,ds^2\]
therefore
\[\norm{d_{(\lambda\,s)}u}^2=
\norm{\frac{\partial u}{\partial\lambda}(\lambda\,s)}^2
+\frac{1}{\lambda^2}
\norm{\frac{\partial u}{\partial s}(\lambda\,s)}^2\]
since we are taking the dual norm in the cotangent bundle.
The partial derivatives are
\[\frac{\partial u}{\partial\lambda}(\lambda\,s)=
(1-\kappa)\lambda^{-\kappa}v(s)\;\;\textrm{ and }\;\;
\frac{\partial u}{\partial s}(\lambda\,s)=
\lambda^{1-\kappa}d_sv\]
thus the Hamilton-Jacobi equation \eqref{HJ} can be written
\[(1-\kappa)^2\lambda^{-2\kappa}v(s)^2+
\lambda^{-2\kappa}\norm{d_sv}^2=2\,U(\lambda\,s)\,.\]
Since $U$ is homogeneous of degree $-2\kappa$, the equation
in $v$ is the Hamilton-Jacobi equation \eqref{HJH}
\begin{equation*}
(1-\kappa)^2\,v(s)^2+\norm{d_sv}^2=2\,U(s)\,.
\end{equation*}

\subsection{Viscosity solutions}\label{visco}

We start this section recalling briefly the well known
notion of viscosity solution of a first order Hamilton-Jacobi
equation of the form
\begin{equation}\label{GeneralHJ}
\H(x,d_xu,u)=0
\end{equation}
introduced by \mbox{M. Crandall},
\mbox{L. Evans} and \mbox{P.-L. Lions}
(see for instance \cite{Crandall-Lions}, \cite{Crandall-Evans-Lions}).
We will assume that $\H$, the Hamiltonian,
is a continuous function defined
on $T^*M\times\R$ where $M$ is a compact smooth manifold,
and moreover, that $\H$ is smooth outside a singular set
of the form $T^*\Delta\times\R$ where $\H=+\infty$
($\Delta\subset M$).

\begin{definition}\label{viscodef}
Let $u:M\to\R$ be a continuous function, and $x_0\in M$.
\begin{itemize}
  \item[-] $u$ is a \emph{viscosity subsolution of
           \eqref{GeneralHJ} at $x_0$} if for every
           $\varphi\in C^1(M)$ such that $\varphi(x_0)=u(x_0)$
           and $\varphi\geq u$ in a neighborhood of $x_0$ we have
           \begin{equation*}
             \H(x,d_{x_0}\varphi,\varphi(x_0))\leq 0
           \end{equation*}
  \item[-] $u$ is a \emph{viscosity supersolution of
           \eqref{GeneralHJ} at $x_0$} if for every
           $\varphi\in C^1(M)$ such that $\varphi(x_0)=u(x_0)$
           and $\varphi\leq u$ in a neighborhood of $x_0$ we have
           \begin{equation*}
             \H(x,d_{x_0}\varphi,\varphi(x_0))\geq 0
           \end{equation*}
  \item[-] $u$ is a \emph{viscosity solution of
           \eqref{GeneralHJ} at $x_0$} if it is both
           viscosity subsolution and viscosity supersolution
           at $x_0$.
  \item[-] $u$ is a \emph{viscosity solution of
           \eqref{GeneralHJ}} if it is a viscosity solution at
           each point $x_0\in M$.
\end{itemize}
\end{definition}

\begin{theorem}\label{existsViscSol2}
Let $u\in\calH$ be an homogeneous weak KAM solution of the $N$-body problem,
and $v$ the restriction of $u$ to the unit sphere $S$.
Then $v$ is a viscosity solution of the Hamilton-Jacobi equation
\eqref{HJH}.
\end{theorem}

\begin{proof}
We know that $u$ is a viscosity solution of Hamilton-Jacobi equation \eqref{HJ}.
Suppose that $\varphi\in C^1(S)$ is such that $\varphi\geq v$ and that
$\varphi(s)=v(s)$. If $\psi$ is the homogeneous extension of $\varphi$
to $V$ of degree $1-\kappa$, then we have that $\psi\geq u$ in a
neighborhood of $s$ and
$\psi(\lambda s)=u(\lambda s)=\lambda^{1-\kappa}\varphi(s)$
for all $\lambda>0$. We also have that
\[\frac{\partial\psi}{\partial s}(s)=d_s\varphi\,.\]
Thus, since $u$ is a viscosity subsolution at $s$ we have
\[\norm{d_s\psi}^2=(1-\kappa)^2\varphi(s)^2+
\norm{d_s\varphi}^2\leq 2\,U(s)\]
and we conclude that $v$ is a viscosity subsolution of \eqref{HJH} at $s$.
A similar argument proves that $v$ is also a viscosity supersolution.
\end{proof}

\subsection{Calibrating curves of homogeneous solutions}

Weak KAM solutions come with a lamination of calibrating
curves, as it was explained in proposition \ref{calibrating} above.
We start showing that the homogeneity of a weak KAM solution
implies an invariance
property of such calibrating curves.

\begin{lemma}\label{lema:calibrating of homogeneous}
If a weak KAM solution $u$ is homogeneous then
the set of calibrating curves is invariant under the
action of $\R^+$ given by
\[\gamma\mapsto\gamma_\lambda\;\;\;
\gamma_\lambda(t)=\lambda\,\gamma(\lambda^{-(1+\kappa)}t)\]
for any $\lambda>0$.
\end{lemma}

\begin{proof}
Suppose that $u$ is homogeneous and that
$\gamma:[0,+\infty)\to V$ calibrates $u$.
Fix $\lambda>0$, and note that the curve $\gamma_\lambda$
is also defined in $[0,+\infty)$.
If we write $x=\gamma(0)$ and $y=\gamma(t)$ for some
value of $t>0$, we have that
\[u(x)-u(y)=A(\gamma\mid_{[0,t]})\,.\]
On the other hand, if we write $t^*=\lambda^{(1+\kappa)}\,t$
we have that
\[\gamma_\lambda(0)=\lambda x\;\;,\;\;\;\;
\gamma_\lambda(t^*)=\lambda y\]
and
\[A(\gamma_\lambda\mid_{[0,t^*]})=
\lambda^{1-\kappa}\,A(\gamma\mid_{[0,t]})\,.\]
Therefore, since $u$ is homogeneous of degree $1-\kappa$,
we conclude that
\[u(\lambda x)-u(\lambda y)=\lambda^{1-\kappa}(u(x)-u(y))
=\lambda^{1-\kappa}\,A(\gamma\mid_{[0,t]})
=A(\gamma_\lambda\mid_{[0,t^*]})\]
which proves that the curve $\gamma_\lambda$ is also calibrating.
\end{proof}

We will denote $\pi:V\setminus\set{0}\to S$ the projection
$\pi(x)=I(x)^{-1/2}\,x$, and $\Omega$ will denote the open and dense set
of configurations without collisions.

\begin{theorem}\label{rho-sigma-eq}
Let $u\in\calH$ be an homogeneous weak KAM solution of the $N$-body problem,
and $v$ the restriction of $u$ to the unit sphere $S$. If
$\gamma:(a,b)\to \Omega$ is a calibrating curve for $u$,
$\rho=I(\gamma)^{1/2}$ and
$\sigma=\pi\circ\gamma$, then for all $t\in (a,b)$ we have that
$v$ is differentiable on $\sigma(t)$ and

\begin{enumerate}
  \item $\dot\rho=(1-\kappa)\rho^{-\kappa}v(\sigma)$\vspace{2mm}
  \item $d_{\sigma(t)}v(\nu)=\inner{\nu}{\rho(t)^{\kappa}\dot\sigma(t)}
\;\;\textrm{ for all }\nu\in T_{\sigma(t)}S$
\end{enumerate}
\end{theorem}

\noindent
\textbf{Remark.} The condition $\gamma(t)\in \Omega$ is needless when Marchal's
theorem applies (for instance in the Newtonian case) because calibrating curves
are always minimizers and must avoid collisions.

\begin{proof}
Suppose now that $\gamma:(a,b)\to \Omega$ calibrates
the homogeneous function $u$. At each $t\in (a,b)$ we have that
$U(\gamma(t))<+\infty$. Thus $\gamma$ is an extremal without collisions
of the Lagrangian action, hence $\gamma$ is smooth.
Since $t\in (a,b)$ is an interior point, $u$ is differentiable at
$\gamma(t)$ and the calibrating condition implies that
$d_{\gamma(t)}u$ is the Legendre transform of $\dot\gamma(t)$
(see \cite{Fathibook}).
In other words, using the mass inner product we have
\[d_{\gamma(t)}u(\xi)=\inner{\xi}{\dot\gamma(t)}\]
for any $\xi\in V$. By homogeneity, $u$ is differentiable at
$\lambda\gamma(t)$ for all $t\in(a,b)$ and any $\lambda>0$.
In polar coordinates we can write
$u(\lambda\,s)=\lambda^{1-\kappa}v(s)$,
thus $v$ is differentiable at
$\sigma(t)$ for all $t\in(a,b)$.
Also using polar coordinates we can write
$\gamma(t)=\rho(t)\sigma(t)$
where $\rho(t)=I(\gamma(t))^{1/2}$ and $\sigma(t)=\pi(\gamma(t))$.
At each time $t\in(a,b)$ a vector $\xi\in V$
can be written as $\xi=r\,\sigma(t)+\nu$ with
$\inner{\nu}{\sigma(t)}=0$. Thus we have
\begin{eqnarray*}
d_{\gamma(t)} u(\xi)&=&
\inner{r\sigma(t)+\nu}{\dot\rho(t)\,\sigma(t)+\rho(t)\,\dot\sigma(t)}\\
&=&r\,\dot\rho(t)+\rho(t)\inner{\nu}{\dot\sigma(t)}
\end{eqnarray*}
and also
\[d_{\gamma(t)}u(\xi)=(1-\kappa)\rho(t)^{-\kappa}v(\sigma(t))\,r+
\rho(t)^{1-\kappa}d_{\sigma(t)}v(\nu)\,.\]
Since $\xi\in V$ is arbitrary,  we have that
\[\dot\rho=(1-\kappa)\rho^{-\kappa}v(\sigma)\]
and that for all $t\in(a,b)$
\[d_{\sigma(t)}v(\nu)=\inner{\nu}{\rho(t)^{\kappa}\dot\sigma(t)}
\;\;\textrm{ for all }\nu\in T_{\sigma(t)}S\]
\end{proof}

\begin{corollary}\label{v-growsalongcalibrating}
Let $u\in\calH$ be an homogeneous weak KAM solution of the $N$-body problem,
and $v$ the restriction of $u$ to the unit sphere $S$. If
$\gamma:(a,b)\to \Omega$ is a calibrating curve for $u$, then
$v(\pi(\gamma))$ is strictly increasing unless $\gamma$ is homothetic.
\end{corollary}

\begin{proof}
If $\gamma=\rho\,\sigma$ as before, then theorem \ref{rho-sigma-eq} implies
\[\frac{d}{dt}v(\sigma(t))=
\inner{\dot\sigma(t)}{\rho(t)^{\kappa}\dot\sigma(t)}=
\rho(t)^\kappa\,\norm{\dot\sigma(t)}^2\geq 0\,.\]
Therefore $v(\sigma)$ is not decreasing. On the other hand, if
we have $v(\sigma(c))=v(\sigma(d))$ for some values of $c<d$
then
\[0=\int_c^d\,\rho(t)^\kappa\,\norm{\dot\sigma(t)}^2\,dt\]
which implies $\dot\sigma(t)=0$ for all $t\in[c,d]$ because $\rho>0$.
We deduce that $\gamma$ is homothetic on the interval $[c,d]$. It is
clear that this implies that $\gamma$ is homothetic over
whole his domain (and that $\sigma$ is a central configuration).
\end{proof}

\section{Minimizing configurations}

As we have say, a minimizing configuration is a central configuration
such that his parabolic ejection is a free time minimizer.
The parabolic ejection of a central configuration $s$ is a curve
$\gamma_s:[0,+\infty)\to V$ of the form
\begin{equation}\label{parabolicray}
\gamma_s(t)=\alpha_s\,t^{c_\kappa}\,s
\end{equation}
where $\alpha_s$ and $c_\kappa$ are positive constants which depends
on the subscripts. We need to compute explicitly these constants.
For the sake of simplicity the configuration $s$ will be supposed of
unit norm, that is, $s\in S$.
In order to be a motion the curve $\gamma_s$, it must satisfies the
Newton's equation of motion $\ddot\gamma_s=\nabla U(\gamma_s)$
(we recall that the gradient is taken with respect to the mass inner product).
The explicit computation gives
\[\ddot\gamma_s(t)=c_\kappa(c_\kappa-1)\alpha_s\,t^{c_\kappa-2}\,s\]
and
\[\nabla U(\gamma_s(t))=
\alpha_s^{-(2\kappa+1)}\,t^{-c_\kappa(2\kappa+1)}\,\nabla U(s)\,.\]
Therefore, the equation of motion will be satisfied if
and only if\vspace{2mm}

\begin{itemize}
  \item[(a)] $\nabla U(s)=\lambda\,s$ for some constant $\lambda$,\vspace{2mm}
  \item[(b)] $c_\kappa-2=-c_\kappa(2\kappa+1)$, and\vspace{2mm}
  \item[(c)] $\lambda\,\alpha_s^{-(2\kappa+1)}=
             c_\kappa(c_\kappa-1)\alpha_s$.\vspace{2mm}
\end{itemize}

The first condition says that $s$ is a central configuration, or
in other words, that $s$ is a critical point of the restriction of
$U$ to the sphere $S$. We will see that if condition (a) is satisfied,
then the equations (b) and (c) have unique solution, namely, the ones
that make the curve $\gamma_s$ a zero energy motion.

Since $U$ is homogeneous of degree $-2\kappa$, the Euler's theorem
gives
\[\inner{\nabla U(s)}{s}=-2\kappa\,U(s)\,.\]
Thus, if (a) is satisfied, we must have
$-2\kappa\,U(s)=\inner{\lambda\,s}{s}=\lambda$.
From (b) we deduce that
\[c_\kappa=\frac{1}{1+\kappa}\]
which takes the well known value $2/3$ in the Newtonian case.
Finally, replacing the found values of $\lambda$ and $c_\kappa$
in equation (c) we get that
\[\alpha_s=\left(2(1+\kappa)^2\,U(s)\right)^{1/2(1+\kappa)}\]
which takes the well known value $(9\,U(s)/2)^{1/3}$ in the Newtonian case.
If now we compute the kinetic energy and the potential function
\[T(t)=\frac{1}{2}\norm{\dot\gamma_s(t)}^2
\;\;\textrm{ and }\;\;U(t)=U(\gamma_s(t))\]
we obtain
\begin{equation}\label{kinetikofparabolicray}
T(t)=U(t)=\;2^{-\kappa/(1+\kappa)}\;(1+\kappa)^{-2\kappa/(1+\kappa)}\;
U(s)^{1/(1+\kappa)}\;t^{-2\kappa/(1+\kappa)}
\end{equation}

which shows that $\gamma_s$ is a zero energy motion.

Because of lemma \ref{homogeneityPHIxyt}, we can deduce
that the critical potential action $\phi(x,y)$ is homogeneous
of degree $1-\kappa$, that is, we have
\begin{equation}\label{phihomogeneo}
\phi(\lambda\,x,\lambda\,y)=\lambda^{1-\kappa}\,\phi(x,y)
\end{equation}
for any pair of configurations $x,y\in V$ and any value of
$\lambda>0$.
Since the curve $\gamma_s$ is invariant under the blow-up
transformations used in lemma \ref{homogeneityPHIxyt} or in
lemma \ref{lema:calibrating of homogeneous}, it is easy to  see
that if the equality
\[A(\gamma_s\mid_{[0,t]})=\phi(0,\gamma_s(t))\]
holds for some $t_0>0$, then it must also hold for every $t>0$.
In particular, since the restriction of a free time minimizer
to a subinterval is also a free time minimizer, if the previous
equality holds for some $t_0>0$ then we will have
\[A(\gamma_s\mid_{[a,b]})=\phi(\gamma_s(a),\gamma_s(b))\]
for every compact interval $[a,b]\subset [0,+\infty)$.
This is precisely the condition required to be the parabolic
ejection $\gamma_s$ a free time minimizer.

We now introduce an auxiliary function $\psi:S\to\R$.
Given normal configuration $s\in S$, let $t(s)>0$
be the time in which the curve $\gamma_s$ above defined
\eqref{parabolicray} pass
through the configuration $s$, that is such that $\gamma_s(t(s))=s$, and set
\begin{equation}\label{funcionaux}
\psi(s)=A(\gamma_s\mid_{[0,t(s)]})\,.
\end{equation}
Of course, we have $\psi(s)\geq \phi(s,0)$ for all $s\in S$.
The above discussion shows that if $\psi(s)=\phi(x,0)$ then the
parabolic ejection $\gamma_s$ is a free time minimizer.
Therefore we have proved the following proposition,
which gives a characterization of the set of normal
minimizing configurations.

\begin{proposition}\label{prop: min iff psi=phi0}
A normal configuration $s\in S$ is minimizing configuration if
and only if satisfies $\psi(s)=\phi(s,0)$.
\end{proposition}

\begin{corollary}\label{minimizing is compact}
The set of normal minimizing configurations $\calM\subset S$ is compact.
\end{corollary}

Let us compute the auxiliary function $\psi$. First we need to
compute the time $t(s)$ of a given configuration $s\in S$.
Clearly we have $\gamma_s(t(s))=s$ if and only if
$\alpha_s\,t(s)^{c_\kappa}=1$, thus we deduce that
\[t(s)=\alpha_s^{-(1+\kappa)}=
\left(2\,(1+\kappa)^2\,U(s)\right)^{-1/2}\,.\]
Since $T(t)=U(t)$ for any $s\in S$,
even is the configuration $s$ is not minimizing, we can write
\[\psi(s)=\int_0^{t(s)}T(t)+U(t)\,dt=2\,\int_0^{t(s)}U(t)\,dt\]
which gives
\[\psi(s)=\frac{1}{1-\kappa}\,(2\,U(s))^{1/2}\]
and in the Newtonian case takes the value $2\sqrt{2U(s)}$.
It is not surprising to note that the explicit computation of
the auxiliary function $\psi$ gives (modulo a square root) one
of the terms of the Hamilton-Jacobi equation \eqref{HJH}.
Using the function $\psi$ we can reformulate
equation \eqref{HJH} in a more suggestive way as
\begin{equation}\label{HJH2}
v(s)^2+\frac{1}{(1-\kappa)^2}\norm{d_sv}^2=\psi(s)^2
 \end{equation}

Another important role of the minimizing configurations is that
they allow to define the critical Busemann functions, since
their associated parabolic ejections are minimizing geodesics
(geodesic rays) for the Jacobi metric on the zero energy level.
More precisely, if $s\in S$ is a minimizing configuration,
the corresponding Busemann critical function is defined by
\[b_s(x)=
\lim_{t\to+\infty}\left(\phi(x,t\,s)-\phi(t\,s,0)\right)\,.\]
These functions are homogeneous weak KAM solutions and are studied
by Boris Percino and Héctor Sánchez-Morgado
in \cite{BorisHector}, where it is proved that the above limit defines a
weak KAM solution and that his calibrating curves are asymptotic to
the minimizing configuration.

\section{Smooth homogeneous solutions.}\label{sec:smoothHsol}

This section is devoted to the Newtonian case $2\kappa=1$ in
a space of dimension at least two. The reason of this restriction
is that we want to apply several results which are until now
only proved for the Newtonian potential
like \cite{AlbouyKaloshin,daLuzMad,Shub}
or which are not true in the collinear case, like
\cite{ChencinerICM,FerrTerr,Marchal}.

The main application of the analysis developed
in the previous sections is the following theorem.
Recall that the unit sphere $S\subset V$ is a Riemannian manifold
as a submanifold of $V$ endowed with the mass inner product.
We will denote $K\subset S$ the compact set of normal configurations
with partial collisions.

\begin{theorem}\label{teo: foliation in S}
Let $U$ be the Newtonian potential ($\kappa=1/2$) and suppose that
$\dim E\geq 2$.
Let $u:V\to\R$ be an homogeneous smooth solution of Hamilton-Jacobi equation
\eqref{HJ} (in the sense of \ref{def:smoothsol})
and $v$ the restriction of $u$ to the unit sphere
$S$. Let $\nabla v$ be his gradient vector field, which is a smooth
vector field on $S\setminus K$.
For $s\in S\setminus K$, let $\theta_s:(a_s,b_s)\to S$ be the maximal
solution of $\dot\theta=\nabla v(\theta)$ with $\theta_s(0)=s$.
Let $Z_v=\set{s\in S\mid \nabla v(s)=0}$.
Then we have:

\begin{enumerate}
  \item[(a)] $Z_v$ is a subset of $\calM$,
             the set of minimizing configurations.
             \vspace{2mm}
  \item[(b)] If $s\in A=S\setminus(K\cup Z_v)$ then\vspace{2mm}
      \begin{enumerate}
        \item[(i)] $a_s=-\infty$ and the $\alpha$-limit set
                   satisfies $\alpha(s)\subset Z_v$.\vspace{2mm}
        \item[(ii)] $b_s<+\infty$ and there is $r(s)\in K$ such that \[\lim_{t\to b_s}\theta(s)=r(s)\,.\]
        \item[(iii)] The map
                     $r:A\to K$ continuous
                     and surjective.
      \end{enumerate}
\end{enumerate}
\end{theorem}

\begin{proof}
Let $s\in Z_v$ and
Since $v$ is the restriction of $u$ to the
unit sphere $S$, then it is clear that we have
\[\abs{v(s)}=\abs{u(s)-u(0)}\leq\phi(s,0)\leq\psi(s)\,.\]
Therefore, if $v$ has a critical point at some configuration
$s\in S$, then equation \eqref{HJH2} implies that
$\abs{v(s)}=\psi(s)$. Thus we must have $\psi(s)=\phi(s,0)$,
which implies that $s$ is a minimizing configuration as a
consequence of proposition \ref{prop: min iff psi=phi0}.
Thus we have proved item (a).

Suppose now that $s\in S\setminus K$ and that $\nabla v(s)\neq 0$.
Let $\gamma:(-\infty,0]$ be the unique calibrating curve for $u$
such that $\gamma(0)=s$ (the uniqueness is ensured by the
differentiability of $u$ at $s$). We know that $\gamma$ is a free
time minimizer, and it is proved in \cite{daLuzMad} that
$\gamma(t)$ is therefore completely parabolic for $t\to -\infty$.
It is known that the normalized configuration of such motions
tends to the set of central configurations. A simple proof
of this fact for homogeneous potential was written by Alain Chenciner
(see ``Théorème fondamental'' in \cite{Chenciner}). On the other
hand, a well known theorem
of Shub \cite{Shub} says that the set of normal central configurations
is a compact subset of $S\setminus K$. Of course this is obvious when
the conjecture of finiteness of the set of similarity classes of
central configurations holds,
like in the three body problem, or in many other cases
(see for instance the work of Albouy and Kaloshin
\cite{AlbouyKaloshin} and the references therein).
We have thus that $\sigma(t)=\pi(\gamma(t))$ tends to the set of
central configurations when $t\to -\infty$.

By theorem \ref{rho-sigma-eq} (2) we have, for each $t\leq 0$,
\[\dot\sigma(t)=\rho(t)^{-1/2}\nabla v(\sigma(t))\]
where $\rho(t)=I(\gamma(t))^{1/2}$. Thus $\sigma$ is a reparametrization
of a segment of $\theta$. If $\sigma(t)=\theta(\tau(t))$ then we
have that $\tau(0)=0$ and that
\[\dot\sigma(t)=\dot\tau(t)\,\dot\theta(\tau(t))=
\dot\tau(t)\,\nabla v(\sigma(t))\]
for all $t\leq 0$. Therefore $\tau$ satisfies
$\dot\tau(t)=\rho(t)^{-1/2}$. By integration we get
\[\tau(t)=-\int_t^0\rho(u)^{-1/2}du\,.\]
In a completely parabolic motion all mutual distances grow like $t^{2/3}$,
thus we have $\rho(u)\sim \abs{u}^{2/3}$ for $u\to -\infty$.
Therefore $\tau(t)\sim -\abs{t}^{2/3}$ for $t\to -\infty$, which proves
that $a_s=-\infty$ and that $\theta(t)$ tends to the compact set
of central configurations for $t\to -\infty$. In particular,
the $\alpha$-limit set $\alpha(s)$ is a well defined compact
connected set of central configurations. Since each point in
$\alpha(s)$ is recurrent, and regular orbits of a gradient flow
are never recurrent, we conclude that each point in $\alpha(s)$
is an equilibrium point, meaning that $\alpha(s)\subset Z_v$.
The statement b.(i) is thus proved.

We will prove now statements b.(ii) and b.(iii). As before, let
$\gamma:(-\infty,0]$ be the unique calibrating curve for $u$
such that $\gamma(0)=s$. By lemma \ref{lema: calibrating smooth newton}
bellow, $\gamma$ can be extended to a maximal motion over an interval
$(0,T_s)$ with $T_s>0$ which also be called $\gamma$.
Moreover, this extension $\gamma$ also calibrates $u$.

Since the extended curve $\gamma$ is maximal,
We conclude that either $\gamma$ presents a pseudocollision
at time $t=T_s$ or there is a collision configuration
$c_s\in V$ such that
$\lim_{t\to T_s}\gamma(t)=c_s$. Recall that
Painlevé has proved that pseudocollisions can only occur
when the number of bodies is at least $N\geq 4$. If $\gamma$
has a collision at time $t=T_s$
there are two possibilities: either $c_s$ is a total collision,
either $c_s$ is a partial collision.

We discuss now these three cases.

\vspace{2mm}
\noindent
\emph{First case: $c_s$ is a total collision.}
  As before, we write $\gamma(t)=\rho(t)\sigma(t)$ for $t\leq T_s$
  where $\sigma=\pi(\gamma)$ and $\rho=I(\gamma)^{1/2}$.
  Writing $\sigma(t)=\theta(\tau(t))$ we have again that
  $\dot\tau(t)=\rho(t)^{-1/2}$ and that $\tau(0)=0$, but this time we know that
  \[\rho(t)\sim \left(T_s-t\right)^{2/3}\]
  for $t\to T_s$ because $\gamma$ presents a total collision at time $T_s$
  (see also \cite{Chenciner}).
  Thus by integration we get again, for $t\in[0,T_s)$
  \[\tau(t)=\int_0^t\rho(u)^{-1/2}du\,.\] If $\alpha>0$ is such that
  $\rho(u)\geq\alpha(T_s-u)^{2/3}$ we deduce the upper bound
  \[\tau(t)\leq \alpha\,\frac{3}{2}\,T_s^{2/3}\]
  for all $t\in (0,T_s)$, and that the limit $\tau_0$ of $\tau(t)$ for
  $t\to T_s$ exists. Now we use again the fact that $\gamma$ presents
  a total collision at time $T_s$ in order to guarantee that
  $\sigma(t)$ tends to the set of central configurations when
  $t\to T_s$. Since $\sigma(t)=\theta(\tau(t))$ and the set of central
  configurations is a compact subset of $S\setminus K$ we have that
  \[\lim_{t\to T_s}\theta(\tau(t))=\lim_{\tau\to\tau_0}\theta(\tau)=
  \theta(\tau_0)=s_0\] where $s_0$ is some central configuration.
  Thus we have proved that $\sigma(t)$ \emph{converges} to a central configuration
  $s_0$. Of course we have that $\nabla v(s_0)\neq 0$ since the
  vector field $\nabla v$ is uniquely integrable in $S\setminus K$.

  On the other hand, using a blow-up technique we can prove that
  the parabolic collision by $s_0$ calibrates $u$, which implies that
  $\nabla v(s_0)=0$. To see this, we start by translating the domain
  of the calibrating curve $\gamma$ in order to have the total
  collision at time $t=0$. Therefore we can suppose that
  $\gamma:(-\infty,0]$ is a calibrating curve of $u$,
  that $\gamma$ has total collision at $t=0$,
  and that $\gamma$ is completely parabolic for $t\to -\infty$.
  Moreover, we have that his normalized shape has a limit, that is
  to say that $\lim_{t\to 0} \sigma(t)=s_0$.
  Now we apply lemma \ref{lema:calibrating of homogeneous}
  to obtain a family of calibrating curves $(\gamma_\lambda)_{\lambda>0}$
  with exactly the same properties. Recall that $\gamma_\lambda$
  is defined for $t\leq 0$ by
  \[\gamma_\lambda(t)=\lambda\,\gamma(\lambda^{-3/2}t)\,.\]
  It is not difficult to prove with well known arguments that
  $\gamma_\lambda$ converges uniformly on compact subsets for
  $\lambda\to +\infty$ to an homothetic curve $\gamma_0$ by $s_0$.
  For instance, restricting the curves to $[-1,0]$, we have that
  \[A(\gamma_\lambda\mid_{[-1,0]})
  =u(\gamma_\lambda(-1))-u(0)=\phi(\gamma_\lambda(-1),0)\,.\]
  Since the Lagrangian action is lower semicontinuous,
  \[\lim_{\lambda\to +\infty}\gamma_\lambda(t)=\gamma_0(t)\]
  uniformly in $t\in [-1,0]$, and the potential action as well
  as the function $u$ are continuous,
  we also have that
  \[A(\gamma_0\mid_{[-1,0]})\leq
  \lim_{\lambda\to +\infty}A(\gamma_\lambda\mid_{[-1,0]})=
  u(\gamma_0(-1))-u(0)\,.\]
  Therefore the curve $\gamma_0$ is an homothetic calibrating
  curve of $u$, and as such it must be the parabolic collision
  by $s_0$. In particular we have $\nabla v(s_0)=0$, and the possibility
  of $c_s$ to be a total collision is excluded.

  \emph{Note that we have proved the following}: Any calibrating curve of an
  homogeneous weak KAM solution with a total collision is homothetic.

\vspace{2mm}
\noindent
\emph{ Second case: $c_s$ is a partial collision.}
  Once again we write $\gamma(t)=\rho(t)\sigma(t)$ for $t\leq T_s$
  where $\sigma=\pi(\gamma)$ and $\rho=I(\gamma)^{1/2}$.
  Writing $\sigma(t)=\theta(\tau(t))$ we have again that
  $\dot\tau(t)=\rho(t)^{-1/2}$, that $\tau(0)=0$ hence that
  \[\tau(t)=\int_0^t \rho(u)^{-1/2}du\]
  for all $t<T_s$. Since
  $\lim_{t\to T_s}\rho(t)=\rho_0=I(c_s)^{-1/2}>0$,
  we can say that $\tau(t)$ tends to the convergent integral
  \[\tau_0=\int_0^{T_s} \rho(u)^{-1/2}du\]
  when $t\to T_s$. Moreover, if $r(s)=\rho_0^{-1}c_s$
  is the normalized configuration of $c_s$, then $r(s)\in K$
  and we have that
  \[r(s)=\lim_{t\to T_s}\sigma(t)=\lim_{\tau\to\tau_0}\theta(\tau)\,.\]
  Note that this implies that the maximal solution $\theta$
  is defined until $\tau_0>0$ so we deduce that $b_s=\tau_0<+\infty$,
  and that in this second case statement b.(ii) holds.

\vspace{2mm}
\noindent
\emph{Third case: $\gamma$ has a pseudocollision at time $t=T_s$}
  We will exclude this possibility. In that case
  $\gamma(t)$ has no limit for $t\to T_s$. As before we write
  in polar coordinates $\gamma=\rho\,\sigma$. Our first step, will be to
  prove that $\rho(t)$ is bounded in $[0,T_s)$. By theorem \ref{rho-sigma-eq}
  we know that $2 \dot\rho(t)=\rho(t)^{-1/2}v(\sigma(t))$. Since
  $v$ is continuous in $S$, we deduce that
  there exists a positive constant $M>0$ for which
  \begin{equation}\label{eq: rho bounded unif}
  \frac{d}{dt}\,\rho(t)^{3/2}=
  \frac{2}{3}\,\rho(t)^{1/2}\dot\rho(t)\leq M
  \end{equation}
  for all $t\in [0,T_s)$. Thus $\rho(t)$ is bounded, and $\gamma(t)$ must
  be contained in a compact subset of $V$ for all $t\in [0,T_s)$.
  Since we are assuming that $\gamma(t)$ has no limit, we must have at least
  two limits points for $t\to T_s$. This is impossible because $\gamma$ is
  a free time minimizer.

\vspace{2mm}
It remains to prove statement b.(iii). Let $s_0\in K$ and $\gamma:(-\infty,0]$
be a calibrating curve of $u$ such that $\gamma(0)=s_0$. It is clear that
for every $t<0$ the configuration $s=\pi(\gamma(t))$ is in $A$ an that
$r(s)=s_0$. Therefore the map $r:A\to K$ is surjective.

In order to prove the continuity, note that it suffices to prove the
continuity of the map $s\mapsto c_s$ since $r(s)=\pi(c_s)$.
Let $(s(n))_{n>0}$ be a convergent sequence in $A$ such that
$\lim s(n)=s\in A$. For each $n>0$, we know that there is a configuration
with partial collisions $c_n$ and a positive time $T_n>0$ with
the following property: the maximal calibrating curve of $u$ passing by
$s_n$ at time zero is a curve $\gamma_n:(-\infty,T_n]\to V$ such that
$\gamma_n(T_n)=c_n$. Now observe that for each $n>0$ we have that
$\gamma_n$ is differentiable at $t=0$ and that the Legendre transform of
$\dot\gamma_n(0)$ is precisely $d_{s(n)}u$ which tends to $d_su$ since $u$ is smooth.
Thus we have that $\dot\gamma_n(0)\to w$, where $w\in V$ is the unique vector
such that his Legendre transform is precisely $d_su$. If $\gamma:(-\infty,T_s]$
is the maximal calibrating curve of $u$ passing by $s$ at time zero, we must
have for the same reason $\dot\gamma(0)=w$, hence we have proved that
$\lim\dot\gamma_n(0)=\dot\gamma(0)$.

\vspace{1mm}
\noindent
\emph{Claim 1 :} $\lim T_n=T_s$.
This is the more delicate part of the proof, and
uses both Tonelli's and Marchal's theorems
(the reader can found the proof of these fundamental theorems
for instance in \cite{Fathibook, Mat} for the first one, and
\cite{ChencinerICM, FerrTerr, Marchal} for the second).
Let $\epsilon>0$ and $0<t<T_s$. The continuity
of the Lagrangian flow on the phase space implies that, for sufficiently
large values of $n>0$ the curves $\gamma_n$ are defined at least
until time $t$, and that
$\norm{\gamma_n(t)-\gamma(t)}\leq\epsilon$. In particular we have
that $\liminf T_n\geq T_s$. Suppose that $\limsup T_n\geq T_s$.
If such is the case, it must exists $\delta>0$ and a subsequence
$(\gamma_{n_k})_{k>0}$ such that each curve $\gamma_{n_k}$ is defined
until time $T_{n_k}>T_s+\delta$. Let now $T^*\in(T_s, T_s+\delta)$.
We know that the sequence $\gamma_{n_k}(T^*)$ is bounded; in fact we
have bounded uniformly these values when we have exclude the case
of pseudocollisions, see the inequality \eqref{eq: rho bounded unif}
above.
Taking again a subsequence, we can assume that
there exists the limit
\[r=\lim_k\,\gamma_{n_k}(T^*)\,.\]
For each $k>0$, the curve $\gamma_n$ is a free time minimizer.
Thus in addition we have,
\[A(\gamma_{n_k}\mid_{[0,T^*]})=\phi(\gamma_{n_k}(T^*),s_n)\]
from which we deduce that
\[A(\gamma_{n_k}\mid_{[0,T^*]})\to \phi(r,s)\,.\]
Therefore Tonelli's theorem applies, and we deduce the existence
of of a subsequence of these curves which converges uniformly
to an absolutely continuous curve $\tilde\gamma$ defined on
$[0,T^*]$.
The lower semicontinuity of the Lagrangian action gives
\[A(\tilde\gamma)\leq \lim A(\gamma_{n_k}\mid_{[0,T^*]})
=\phi(r,s)\]
which says that $\tilde\gamma$ is also a free time minimizer.
The proof of the claim ends as follows: for every $T\in (0,T_s)$ we know that
$\gamma_{n_k}\mid_{[0,T]}$ also converges uniformly to $\gamma_0$.
Thus we have $\tilde\gamma(t)=\gamma(t)$ for every $t\in (0,T_s)$.
We conclude that $\tilde\gamma(T_s)=c_s$ is a configuration of partial
collisions, which contradicts Marchal's theorem since a minimizer
can not present a collision at any interior point of his domain.
Therefore we have proved the claim that $T_n\to T_s$,
but also

\vspace{1mm}
\noindent
\emph{Claim 2 :} Given $\epsilon>0$
and $t<T_s$ we have, for $n>0$ large enough, that $t<T_n$ and that
$\norm{\gamma_n(t)-\gamma(t)}\leq\epsilon$

To complete the proof of the theorem, we now prove that
\[\lim_n c_n=\lim_n\;\lim_{t\to T_n}\,\gamma_n(t)=
\lim_n\;\lim_{t\to T_s}\,\gamma(t)=c_s\,.\]

In what follows, we will assume that $c_n$ does not converges to $c_s$,
and arrive at a contradiction
As before, we use the fact that $\gamma_n(t)$ is uniformly bounded
as consequence of inequality \eqref{eq: rho bounded unif}.
Therefore, we can assume (taking a subsequence
if necessary) that $\lim c_n=c^*\neq c_s$.

Let us call $4\,d=\norm{c^*-c_s}>0$ and fix $\delta>0$.
Let $t<T_s$ be such that $\norm{\gamma(t)-c_s}\leq d$.
Choose $n_0>0$ such that $\norm{c_n-c^*}\leq d$ for every $n>n_0$.
Using now the \emph{claim 2} above with $\epsilon=d$,
we see that we can choose $n(t)>n_0$ such that
$\norm{\gamma_{n(t)}(t)-\gamma(t)}\leq d$. We can also assume
that $n(t)>n_0$ is large enough to have
\[\abs{T_{n(t)}-T_s}\leq\delta/2\,.\]
Therefore we have
\[\norm{c_{n(t)}-\gamma_{n(t)}(t)}\geq \norm{c^*-c_s}-3\,d=d\,\]
for all $t<T_s$ such that $\norm{\gamma(t)-c_s}\leq d$.
Moreover, each $\gamma_n$ is a free time minimizer and
$c_n=\gamma_n(T_n)$ for each $n>0$, therefore we also have
\[A(\gamma_{n(t)}\mid_{[t,T_{n(t)}]})=\phi(c_{n(t)},\gamma_{n(t)}(t))
=\phi(c_{n(t)},\gamma_{n(t)}(t),T_{(n(t)}-t)\,.\]
Let us write, to simplify the notation, $x_t=c_{n(t)}$,
$y_t=\gamma_{n(t)}(t)$ and $\tau_t=T_{n(t)}-t$.
Accordingly we can write $\norm{x_t-y_t}\geq d>0$ and
\[\phi(x_t,y_t,\tau_t)=\phi(x_t,y_t)\,.\]
Note that if $\abs{T_s-t}<\delta/2$ then we have
$\tau_t<\delta$. This is impossible for $\delta$
small enough as consequence of lemma
\ref{lema: lower bound phixyt} bellow.
This ends the proof of the theorem.
\end{proof}

\begin{lemma}\label{lema: calibrating smooth newton}
Suppose that $U$ is the Newtonian potential and that
$\dim E\geq 2$. Let $u:V\to\R$ be a smooth solution of Hamilton-Jacobi
equation \eqref{HJ} in the sense of \ref{def:smoothsol}.
Let $x\in E^N$ be a configuration without collisions and
$\gamma:(-\infty,0]\to V$ be a calibrating curve of $u$ such that $\gamma(0)=x$.
Then the maximal solution of the motion equation of the $N$ bodies
which extends $\gamma$
is defined until a positive finite time $a\in(0,+\infty)$.
\end{lemma}

\begin{proof}
We start recalling that If $\gamma:(a,b)\to V$
is calibrating of a function $u\in\calH$
then, $\gamma$ is a free time minimizer, hence has no collisions as
consequence of Marchal's theorem \cite{Marchal, ChencinerICM}.
In particular $\gamma$ is a differentiable motion of zero energy.
Moreover, $u$ is differentiable at $\gamma(t)$ and the Legendre transform
of $\dot\gamma(t)$ is precisely the derivative of $u$ at $\gamma(t)$
for all $t\in (a,b)$. On the other hand, if we only know that
$\gamma\mid_{(a,t]}$ calibrates $u$ for some $t\in (a,b)$ but we
also know that $u$ is differentiable at $\gamma(t)$, then
$\gamma$ must calibrate $u$ on a bigger interval $(a,t+\epsilon)$
for some $\epsilon>0$ (see \cite{Fathibook}).

Let $\gamma^*:(-\infty,a)\to V$, $a\in\R\cup\set{+\infty}$ be the maximal
motion extending $\gamma$. If $\gamma^*$ is not a calibrating curve
of $u$, then we can define
\[\tau=\max
\set{t\in (0,a)\mid \gamma^*\mid_{(-\infty,t]}\;\textrm{calibrates}\;u}\]
and clearly we have $0<\tau<a$. We deduce that $u$ can not be differentiable
at $\gamma^*(\tau)$ which contradicts the above considerations.
We conclude that $\gamma^*$ is calibrating for $u$ and as such,
it is a free time minimizer. On the other hand, we know that there are no
complete free time minimizers (see \cite{daLuzMad} theorem 1.2).
Therefore we must have $a<+\infty$.
\end{proof}

\begin{lemma}\label{lema: lower bound phixyt}
For every pair of configurations $x,y\in V$ and any $t>0$, we have
\[\phi(x,y,T)\geq \frac{1}{2}\,\norm{x-y}^2\,T^{-1}\,.\]
\end{lemma}

\begin{proof}
Let $\gamma:[0,T]\to V$ be any absolutely continuous curve such that
$\gamma(0)=x$ and $\gamma(t)=y$. Neglecting the integral of the Newtonian
potential in the definition of the Lagrangian action we deduce that
\[2\,A(\gamma)\geq \int_0^T\norm{\dot\gamma(t)}^2 \,dt\,.\]
Applying the Bunyakovsky inequality we can write
\begin{eqnarray*}
\norm{x-y}&\leq&\int_0^T\norm{\dot\gamma(t)}dt\\
&\leq&
\left(\int_0^T\,dt\right)^{1/2}\left(\int_0^T\norm{\dot\gamma(t)}^2\right)^{1/2}\\
&\leq& T^{1/2}\,(2\,A(\gamma))^{1/2}
\end{eqnarray*}
from which we get
\[2\,A(\gamma)\geq \norm{x-y}^2\,T^{-1}\]
so the proof is obtained taking the infimum over all possible curve $\gamma$.
\end{proof}

\begin{proof}[\bf Proof of theorem \ref{teo:3BP no smooth homogeneous}]

Suppose that there exists a smooth homogeneous solution $u:V\to\R$
of Hamilton-Jacobi equation \eqref{HJ}. Let $v$ the restriction of $u$
to the unit sphere $S$. Theorem \ref{teo: foliation in S} says that
the set $Z_v$ of critical points of $v$ is contained in $\calM$ which
has at most five connected components, namely
three corresponding to Euler configurations, and the corresponding to
the Lagrange equilateral configurations (two components in the planar case
and only one if $k=\dim E\geq3$). On
the other hand, the compact set $K$ of normal partial collisions on $V$
has three connected components, which we will call $K_{12}$, $K_{23}$ and
$K_{31}$. We note that the open set $A=S\setminus(Z_v\cup K)$ is connected.
This is clear for the planar or the spatial three body problem, since
the compact sets $\calM$, and $K_{ij}$ are a finite number of
orbits of the action of the orthogonal group $O(E)$ which has dimension
$k(k-1)/2$, where $k=\dim E$, and $\dim S=2k-1$. Thus we have codimension
$2$ for $k\in\set{2,3}$. But in fact $A$ is connected for every $k\geq 2$,
see proposition \ref{prop: A connected} bellow.
Finally,
applying part b.(iii) of theorem \ref{teo: foliation in S} we conclude that
$K=r(A)$ is connected, which we know to be false.
\end{proof}

We are convinced that the following proposition may be useful to generalize the
application of the here developed techniques to the case of more than three bodies.
When the number of bodies $N\geq 4$ the set $K$ becomes connected, suggesting that
other topological invariants should be considered. Of course, a constructive proof
could be established but surely it would be more cumbersome than presented here.

\begin{proposition}\label{prop: A connected}
For the Newtonian $N$-body problem we have that,
if the set of similarity classes of central
configurations is finite,
then the open set of noncentral
configurations without collisions
is connected.
\end{proposition}

\begin{proof}
Of course we are excluding the collinear case where the
set has $n!$ connected components. Let $x,y\in V$ two
given configurations which are not central
nor collision configurations. It is proven in \cite{daLuzMad}
(theorem 3.1) that there is a free time minimizer $\gamma:[0,T]\to V$
such that $\gamma(0)=x$ and $\gamma(T)=y$. By Marchal's theorem
we know that $\gamma(t)$ has no collisions for $t\in(0,t)$.
Moreover, since the set of central configurations is closed,
we have that $\gamma(t)$ is noncentral if $t$ is close to $0$ or
if $t$ is close to $T$. Let us define
\[C=\set{t\in (0,T)\mid \gamma(t) \textrm{ is a central configuration }}\,.\]
We claim that $C$ is finite. Otherwise $C$ must accumulate at some
$t^*\in(0,T)$. Thus using the finiteness hypothesis we can choose
a sequence
$t_n\to t^*$, and fixed central configuration $z$
such that $\gamma(t_n)$ is similar to $z$ for all
$n>0$. Writing
$\gamma(t)=(r_1(t),\dots,r_N(t))$ and $z=(z_1,\dots,z_N)$
we have that the equalities
\begin{equation}\label{ijkl}
\alpha_{ijkl}(t)=\frac{\norm{r_i(t)-r_j(t)}}{\norm{r_k(t)-r_l(t)}}=
\frac{\norm{z_i-z_j}}{\norm{z_k-z_l}}
\end{equation}
hold for each $t=t_n$ and every choice of $i,j,k,l\in\set{1,\dots,N}$
such that $i\neq j$ and $k\neq l$. Since $\gamma$ has no collisions,
the functions $\alpha_{ijkl}$
are analytic functions of $t$ on $(0,T)$. Therefore each one of the
functions $\alpha_{ijkl}$ is constant, which means that $\gamma(t)$
is similar to $z$ for all $t\in [0,T]$ contradicting the fact that the
configurations $x$ and $y$ are not central.

We have proved that there is at most a finite set of times
$0<t_1<\cdots<t_k<T$ such that $\gamma(t)$ is a central configuration.
We will now perturb the curve $\gamma$ slightly,
thereby avoiding central configurations. This can be achieved perturbing
slightly only one of the functions $r_i(t)$ in small neighborhoods of
times $t=t_i$. If the perturbation $\gamma'$ is sufficiently small, we have for
$t-t_i$ small, that  $\gamma'(t)$ is neither similar to $\gamma(t_i)$ nor any
other central configuration.
\end{proof}

\vspace{3mm}
\noindent
\emph{Acknowledgements.} I would like to express my special
gratitude to M.-C. Arnaud,
S. Terracini and \mbox{A. Venturelli} for motivating me to continue
working in this area, and also to V. Kaloshin for giving me
the opportunity to visit the Department of Mathematics at
the University of Maryland where part of this research was done.

\end{document}